\documentclass{amsart}
\usepackage{amstext,amssymb,amsthm,amsopn,newlfont,graphpap,graphics,graphicx,mathrsfs,enumitem}
\allowdisplaybreaks
\usepackage[parfill]{parskip}
\usepackage[noadjust]{cite}
\usepackage{epigraph}
\usepackage[colorlinks=true,
            linkcolor=red,
            urlcolor=blue,
            citecolor=magenta]{hyperref}
\usepackage{color}
\usepackage{mathrsfs}
\allowdisplaybreaks
\theoremstyle{plain}
\newtheorem{thm}{Theorem}[section]
\newtheorem*{thm*}{Theorem}
\newtheorem{prop}{Proposition}[section]
\newtheorem*{prop*}{Proposition}
\newtheorem{cor}{Corollary}[section]
\newtheorem*{cor*}{Corollary}

\newtheorem*{lem*}{Lemma}
\theoremstyle{definition}

\newtheorem*{defn*}{Definition}
\newtheorem{exmps}{Examples}[section]
\newtheorem*{exmps*}{Examples}
\newtheorem{exmp}[exmps]{Example}
\newtheorem*{exmp*}{Example}

\newtheorem*{exerc*}{Exercise}

\newtheorem{rems}{Remarks}[section]
\newtheorem*{rems*}{Remarks}

\newtheorem*{rem*}{Remark}
\newcommand{\R}{{\mathbb R}}
\newcommand{\C}{{\mathbb C}}

\newcommand{\emps}{\emptyset}

\renewcommand{\bar}{\overline}

\numberwithin{equation}{section}
\DeclareMathOperator{\Rep}{Re\ignorespaces}

\DeclareMathOperator{\dist}{dist}

\DeclareMathOperator*{\esup}{\mbox{$E_A$}-ess\,sup}

\begin{document}
\title[On asymptotics for $C_0$-semigroups]
{On asymptotics for $C_0$-semigroups}
\author[Marat V. Markin]{Marat V. Markin}
\address{
Department of Mathematics\newline
California State University, Fresno\newline
5245 N. Backer Avenue, M/S PB 108\newline
Fresno, CA 93740-8001, USA
}
\email{mmarkin@csufresno.edu}
\begin{abstract}
We stretch the \textit{spectral bound equal growth bound condition} along with a \textit{generalized Lyapunov stability theorem}, known to hold for $C_0$-semigroups of normal operators on complex Hilbert spaces, to $C_0$-semigroups of scalar type spectral operators on complex Banach spaces. For such semigroups, we obtain exponential estimates with the best stability constants. We also extend to a Banach space setting a  celebrated characterization of uniform exponential stability for $C_0$-semigroups on complex Hilbert spaces and thereby acquire a characterization of uniform exponential stability for scalar type spectral and eventually norm-continuous $C_0$-semigroups. 
\end{abstract}
\maketitle

\section[Introduction]{Introduction}

Based on the recently established fact that $C_0$-semigroups of scalar type spectral operators on complex Banach spaces are subject to a \textit{precise weak spectral mapping theorem} \cite{arXiv:2002.09087}, we stretch to such semigroups the \textit{spectral bound equal growth bound condition} along with a \textit{generalized Lyapunov stability theorem} (see Preliminaries), known to hold for $C_0$-semigroups of normal operators on complex Hilbert spaces, and further obtain exponential estimates with the best stability constants for them. We also extend to a Banach space setting the  celebrated \textit{Gearhart-Pr\"uss-Greiner characterization} of uniform exponential stability for $C_0$-semigroups on complex Hilbert spaces \cite[Theorem V.$3.8$]{Engel-Nagel2006} (see \cite{Gearhart1978,Pruss1984,Greiner1985}) and thereby acquire a characterization of uniform exponential stability for scalar type spectral and eventually norm-continuous $C_0$-semigroups. 

\section[Preliminaries]{Preliminaries}

Here, we outline certain essential preliminaries.

\subsection{Resolvent Set and Spectrum}\

For a closed linear operator $A$ in a complex Banach space $X$, the set
\[
\rho(A):=\left\{ \lambda\in \C \,\middle|\, \exists\, 
R(\lambda,A):={(A-\lambda I)}^{-1}\in L(X) \right\}
\]
($I$ is the \textit{identity operator} on $X$, $L(X)$ is the space of bounded linear operators on $X$) and its complement $\sigma(A):=\C\setminus \rho(A)$ are called its \textit{resolvent set} and \textit{spectrum}, respectively.

\subsection{$C_0$-Semigroups}\

A $C_0$-semigroup $\left\{T(t) \right\}_{t\ge 0}$ on a complex Banach space $(X,\|\cdot\|)$ with generator $A$ subject to a \textit{weak spectral mapping theorem}
\begin{equation*}\tag{WSMT}\label{WSMT}
\sigma(T(t))\setminus \{0\}=\overline{e^{t\sigma(A)}}\setminus \{0\},\ t\ge 0,
\end{equation*}
($\overline{\cdot}$ is the \textit{closure} of a set) satisfies the following \textit{spectral bound equal growth bound condition}:
\begin{equation*}\tag{SBeGB}\label{SBeGB}
s(A)=\omega_0,
\end{equation*}
where 
\[
s(A):=\sup\left\{\Rep\lambda \,\middle|\,\lambda\in\sigma(A) \right\}
\ (s(A):=-\infty\ \text{if}\ \sigma(A)=\emptyset)
\] 
is the \textit{spectral bound} of the generator $A$ and
\begin{equation}\label{gb}
\omega_0:=\inf\left\{\omega\in\R \,\middle|\,\exists\,M=M(\omega)\ge 1:\ \|T(t)\|\le M e^{\omega t},\ t\ge 0 \right\}
\end{equation}
(here and henceforth, we use the same notation for the \textit{operator norm} as for the norm on $X$) is the \textit{growth bound} of the semigroup \cite[Proposition  V.$2.3$]{Engel-Nagel2006} (see also \cite{van Neerven1996}).

Generally,
\[
-\infty\le s(A)\le \omega_0 = \inf_{t>0}\frac{1}{t}\ln\|T(t)\|
= \lim_{t\to \infty}\frac{1}{t}\ln\|T(t)\|<\infty
\]
(see, e.g., \cite[Proposition V.$1.22$]{Engel-Nagel2006}).

An \textit{eventually norm-continuous} $C_0$-semigroup $\left\{T(t) \right\}_{t\ge 0}$ on a complex Banach space, i.e., such that, for some $t_0>0$, the operator function
\[
[t_0,\infty)\ni t\mapsto T(t)\in L(X),
\]
is \textit{continuous} relative to the \textit{operator norm}, is subject to the following stronger version of \eqref{WSMT}:
\begin{equation*}\tag{SMT}\label{SMT}
\sigma(T(t))\setminus \{0\}=e^{t\sigma(A)},\ t\ge 0,
\end{equation*}
called a \textit{spectral mapping theorem} (see \cite[Proposition  V.$2.3$]{Engel-Nagel2006} and \cite[Theorem V.$2.8$]{Engel-Nagel2006}). The class of eventually norm-continuous $C_0$-semigroups encompasses $C_0$-semigroups with certain regularity properties, such as \textit{eventually compact} and \textit{eventually differentiable}, in particular \textit{analytic} and \textit{uniformly continuous} (see \cite[Section II.5]{Engel-Nagel2006}).

The asymptotic behavior of a $C_0$-semigroup $\left\{T(t) \right\}_{t\ge 0}$ on a complex Banach space with generator $A$ satisfying \textit{spectral bound equal growth bound condition} \eqref{SBeGB} is governed by the spectral bound of its generator and, in particular, is subject to the subsequent generalization of the classical \textit{Lyapunov Stability Theorem} \cite[Theorem V.$3.6$]{Engel-Nagel2006} (see also \cite{Lyapunov1892}).

\begin{thm}[Generalized Lyapunov Stability Theorem]\label{GLST}\ \\
A $C_0$-semigroup $\left\{T(t) \right\}_{t\ge 0}$ on a complex Banach space $X$ with generator $A$, for which spectral bound equal growth bound condition
\eqref{SBeGB} holds, is uniformly exponentially stable,  i.e., $\omega_0<0$, or equivalently,
\[
\exists\, \omega<0,\ \exists\,M=M(\omega)\ge 1:\ \|T(t)\|\le M e^{\omega t},\ t\ge 0,
\]
iff 
\[
s(A)<0.
\]
\end{thm} 

Cf. \cite[Definition V.$3.1$]{Engel-Nagel2006}, \cite[Proposition V.$3.5$]{Engel-Nagel2006}, and \cite[Theorem V.$3.7$]{Engel-Nagel2006}.

The following statement \cite{Gearhart1978,Pruss1984,Greiner1985} characterizes uniform exponential stability for $C_0$-semigroups on complex Hilbert spaces.

\begin{thm}[Gearhart-Pr\"uss-Greiner Characterization {\cite[Theorem V.$3.8$]{Engel-Nagel2006}}]\label{G-P-GC}\ \\
A $C_0$-semigroup $\left\{T(t) \right\}_{t\ge 0}$ on a complex Hilbert space space $(X,(\cdot,\cdot),\|\cdot\|)$ with generator $A$ is uniformly exponentially stable iff
\begin{equation*}
\left\{\lambda\in \C \,\middle|\, \Rep\lambda\ge 0\right\}\subseteq \rho(A)\quad \text{and}\quad
\sup_{\Rep\lambda\ge 0}\|R(\lambda,A)\|<\infty.
\end{equation*}
\end{thm}

A $C_0$-semigroup $\left\{T(t) \right\}_{t\ge 0}$ (of \textit{normal operators}) on a complex Hilbert space generated by a \textit{normal operator} $A$ is subject to the following precise version of \textit{weak spectral mapping theorem} \eqref{WSMT}:
\begin{equation*}\tag{PWSMT}\label{PWSMT}
\sigma(T(t))=\overline{e^{t\sigma(A)}},\ t\ge 0,
\end{equation*}
\cite[Corollary V.$2.12$]{Engel-Nagel2006} without being a priori eventually norm continuous, and hence, to \textit{spectral bound equal growth bound condition} \eqref{SBeGB} along with the \textit{Generalized Lyapunov Stability Theorem} (Theorem \ref{GLST}). 

\subsection{Scalar Type Spectral Operators}\

A {\it scalar type spectral operator} is a densely defined closed linear operator $A$ in a complex Banach space with strongly $\sigma$-additive \textit{spectral measure} (the \textit{resolution of the identity}) $E_A(\cdot)$, which assigns to the Borel sets of the complex plane projection operators on $X$ and has the operator's \textit{spectrum} $\sigma(A)$ as its {\it support} \cite{Dunford1954,Survey58,Dun-SchIII}.

Associated with such an operator is the {\it Borel operational calculus}, assigning to each Borel measurable function $F:\sigma(A)\to \bar{\C}$ ($\bar{\C}:=\C\cup \{\infty\}$ is the extended complex plane) with $E_A\left(\left\{\lambda\in \C \,\middle|\, F(\lambda)=\infty\right\}\right)=0$ a scalar type spectral operator
\begin{equation*}
F(A):=\int\limits_{\sigma(A)} F(\lambda)\,dE_A(\lambda)
\end{equation*} 
in $X$, whose spectral measure is the image of $E_A(\cdot)$ under $F(\cdot)$, i.e.,
\begin{equation*}
E_{F(A)}(\delta)=E_{A}(F^{-1}(\delta)),\ \delta\in \mathscr{B}(\C),
\end{equation*}
($\mathscr{B}(\C)$ is the \textit{Borel $\sigma$-algebra} on $\C$), with
\begin{equation*}
A=\int\limits_{\sigma(A)} \lambda\,dE_A(\lambda)
\end{equation*}
\cite{Bade1954,Dunford1954,Survey58,Dun-SchIII}.

On a complex finite-dimensional Banach space, scalar type spectral operators are those linear operators, which furnish an \textit{eigenbasis} for the space, i.e., allow a diagonal matrix representation (see, e.g., \cite{Dunford1954,Survey58,Dun-SchIII}). 

In a complex Hilbert space, scalar type spectral operators are those that are similar to {\it normal operators} \cite{Wermer1954} (see also \cite{Lorch1939,Mackey1952}), the latter being the scalar type spectral operators for which the corresponding spectral measure projections are \textit{orthogonal} (see, e.g., \cite{Dun-SchII,Plesner}).

Various examples of scalar type spectral operators, including differential operators arising in the study of linear systems of partial differential equations, in particular perturbed Laplacians, can be found in \cite{Dun-SchIII}.

Due to its strong $\sigma$-additivity, the spectral measure is uniformly bounded, i.e.,
\begin{equation}\label{bounded}
\exists\, M\ge 1\ \forall\,\delta\in \mathscr{B}(\C):\ \|E_A(\delta)\|\le M
\end{equation}
(see, e.g., \cite{Dun-SchI}).


By \cite[Theorem XVIII.$2.11$ (c)]{Dun-SchIII}, for a Borel measurable function $F:\sigma(A)\to \bar{\C}$, the operator $F(A)$ is \textit{bounded} iff $F(\cdot)$ is $E_A$-\textit{essentially bounded}, i.e.,
\[
\esup_{\lambda\in \sigma(A)}|F(\lambda)|<\infty,
\]
in which case
\begin{equation}\label{boundedop}
\esup_{\lambda\in \sigma(A)}|F(\lambda)|\le \|F(A)\|
\le 4M\esup_{\lambda\in \sigma(A)}|F(\lambda)|,
\end{equation}
where $M\ge 1$ is from \eqref{bounded}.

A scalar type spectral $C_0$-semigroup $\left\{T(t) \right\}_{t\ge 0}$ (i.e., a $C_0$-semigroup of scalar type spectral operators) on a complex Banach space $X$ is generated by a scalar type spectral operator \cite{Berkson1966,Panchapagesan1969}, which is the case \textit{iff}
\[
s(A)<\infty
\] 
with
\[
T(t)=e^{tA}:=\int\limits_{\sigma(A)} e^{t\lambda}\,dE_A(\lambda),\ t\ge 0,
\]
{\cite[Proposition $3.1$]{Markin2002(2)}}, the orbit maps of the semigroup
\[
T(t)f=e^{tA}f,\ t\ge 0,f\in X,
\] 
being the \textit{weak solutions} (also called the \textit{mild solutions}) of the associated abstract evolution equation
\begin{equation*}
y'(t)=Ay(t),\ t\ge 0,
\end{equation*}
\cite{Markin2002(1)} (see also \cite{Ball,Engel-Nagel2006}).

\section{SBeGB Condition and Exponential Estimates}

By the \textit{Precise Weak Spectral Mapping Theorem} \cite[Theorem $5.1$]{arXiv:2002.09087}, a $C_0$-semigroup $\left\{T(t) \right\}_{t\ge 0}$ (of scalar type spectral operators) on a complex Banach space generated by a scalar type spectral operator $A$ is subject to \textit{precise weak spectral mapping theorem} \eqref{PWSMT}. Thus, by \cite[Proposition  V.$2.3$]{Engel-Nagel2006} (see Preliminaries), we arrive at the following

\begin{cor}[SBeGB Condition]\label{SBeGBcor}\ \\
For a $C_0$-semigroup $\left\{T(t) \right\}_{t\ge 0}$ (of scalar type spectral operators) on a complex Banach space generated by a scalar type spectral operator $A$, spectral bound equal growth bound condition \eqref{SBeGB} holds.
\end{cor}

\begin{rems}\
\begin{itemize}
\item Considering that, for a scalar type spectral operator $A$ in a complex Banach space, $\sigma(A)\neq \emps$,  when such an operator generates a 
$C_0$-semigroup $\left\{T(t) \right\}_{t\ge 0}$,
\[
-\infty <s(A)=\omega_0<\infty
\]
(see Preliminaries).
\item For a $C_0$-semigroup $\left\{T(t) \right\}_{t\ge 0}$, the \textit{exponential estimates}
\begin{equation*}\tag{EE}\label{EE}
\exists\, \omega\in \R,\ \exists\, M=M(\omega)\ge 1:\ \|T(t)\|\le Me^{\omega t},\ t\ge 0,
\end{equation*}
hold for all $\omega>\omega_0$ with some $M=M(\omega)\ge 1$ (see Preliminaries). 

However, the \textit{best stability constants}, i.e., the smallest numbers $\omega$ and $M$ for which \eqref{EE} is valid, (cf. \cite{Lat-Yur2013}) need not exist, i.e., the infimum in \eqref{gb} may not be attained even when $\omega_0\in \R$ (see \cite[Section I.1]{Engel-Nagel2006} and Example \ref{QNexmp}).
\end{itemize}
\end{rems}

\begin{prop}[Exponential Estimate for Scalar Type Spectral $C_0$-Semigroups]\label{GRESTP}
Let $\left\{T(t) \right\}_{t\ge 0}$ be a $C_0$-semigroup  (of scalar type spectral operators) on a complex Banach space $(X,\|\cdot\|)$ generated by a scalar type spectral operator $A$ with spectral measure $E_A(\cdot)$. Then
\begin{equation}\label{grest}
\|T(t)\|\le 4M_0 e^{\omega_0t},\ t\ge 0,
\end{equation}
where $\omega_0=s(A)$ is the best stability constant in the exponent and 
\begin{equation}\label{bsc}
M_0:=\sup_{\delta \in \mathscr{B}(\C)}\|E_A(\delta)\|\ge 1.
\end{equation}
\end{prop} 

\begin{proof}
Since
\[
T(t)=e^{tA}:=\int\limits_{\sigma(A)} e^{t\lambda}\,dE_A(\lambda),\ t\ge 0,
\]
(see Preliminaries), by \eqref{boundedop}, for any $t\ge 0$,
\begin{multline*}
\left\|T(t)\right\|=\|e^{tA}\|
=\left\|\int\limits_{\sigma(A)} e^{t\lambda}\,dE_A(\lambda)\right\|
\le 4M\esup_{\lambda\in \sigma(A)}\left|e^{t\lambda}\right|
\le 4M\sup_{\lambda\in \sigma(A)}\left|e^{t\lambda}\right|
\\
\shoveleft{
=4M\sup_{\lambda\in \sigma(A)}e^{t\Rep\lambda}\le 4Me^{s(A)t}
}\\
\hfill
\text{in view of $s(A)=\omega_0$};
\\
\ \ \
=4Me^{\omega_0t},
\hfill
\end{multline*}
where $M\ge 1$ is from \eqref{bounded}.

Therefore,
\[
\left\|T(t)\right\|=\|e^{tA}\|\le 4M_0e^{\omega_0t},\ t\ge 0,
\]
where $\omega_0=s(A)$ (see Corollary \ref{SBeGBcor}) is the \textit{best stability constant in the exponent} and $M_0$, defined by \eqref{bsc}, is the \textit{smallest} $M\ge 1$ for which \eqref{bounded} holds.
\end{proof}

\begin{samepage}
\begin{rems}\
\begin{itemize}
\item Thus, for a $C_0$-semigroup $\left\{T(t) \right\}_{t\ge 0}$ (of scalar type spectral operators) on a complex Banach space generated by a scalar type spectral operator $A$ with spectral measure $E_A(\cdot)$,
$\omega_0=s(A)$ is the best stability constant in the exponent and the other best stability constant satisfies the estimate
\[
1\le \min\left\{M\ge 1\,\middle|\,  \|T(t)\|\le Me^{\omega_0t},\, t\ge 0 \right\}\le 4\sup_{\delta \in \mathscr{B}(\C)}\|E_A(\delta)\|.
\]
\item As the next example demonstrates, for a non-scalar-type-spectral $C_0$-semi\-group on a complex Banach space, even under \textit{spectral bound equal growth bound condition} \eqref{SBeGB}, the best stability constants need not exist.
\end{itemize}
\end{rems}
\end{samepage}

\begin{exmp}\label{QNexmp}
On the complex Banach space $l_p^{(2)}$ ($1\le p\le \infty$), the bounded linear operator $A$ of multiplication by the matrix
\[
\begin{bmatrix}
0&1\\
0&0
\end{bmatrix},
\]
which is nonzero \textit{nilpotent}, and hence, not scalar type spectral (see, e.g., \cite{Dun-SchIII}), generates the \textit{uniformly continuous} (not scalar type spectral) semigroup of its exponentials
\[
e^{tA}:=\sum_{n=0}^\infty\dfrac{t^n}{n!}A^n,\ t\ge 0,
\]
where $e^{tA}$, $t\ge 0$, is the bounded linear operator on $l_p^{(2)}$ of multiplication by the matrix
\[
\begin{bmatrix}
1&t\\
0&1
\end{bmatrix}
\]
(see, e.g., \cite{Engel-Nagel2006,Markin2020EOT}).

Although, in the considered case,
\begin{equation*}
s(A)=\omega_0=0,
\end{equation*}
the exponential estimate
\begin{equation*}
\|T(t)\|\le M e^{\omega_0t}=M,\ t\ge 0,
\end{equation*}
holds for no $M\ge 1$ (cf. \cite{Engel-Nagel2006}).
\end{exmp}

For $C_0$-semigroups of normal operators on complex Hilbert spaces exponential estimate \eqref{grest} can be refined as follows.

\begin{prop}[Exponential Estimate for Normal $C_0$-Semigroups]\label{GRESTN}\ \\
Let $\left\{T(t) \right\}_{t\ge 0}$ be a $C_0$-semigroup (of normal operators) on a complex Hilbert space $(X,(\cdot,\cdot),\|\cdot\|)$ generated by a normal operator $A$ with spectral measure $E_A(\cdot)$. Then
\begin{equation}\label{grestn}
\|T(t)\|\le e^{\omega_0t},\ t\ge 0,
\end{equation}
with $\omega_0=s(A)$ and $M_0:=1$ being the best stability constants.
\end{prop} 

\begin{proof}
Since
\[
T(t)=e^{tA}:=\int\limits_{\sigma(A)} e^{t\lambda}\,dE_A(\lambda),\ t\ge 0,
\]
(see, e.g., \cite{Engel-Nagel2006,Plesner}), for arbitrary $t\ge 0$ and $f\in X$, by the properties of the Borel operational calsulus,
\begin{align*}
\|T(t)f\|&=\|e^{tA}f\|
=\left\|\int\limits_{\sigma(A)} e^{t\lambda}\,dE_A(\lambda)f\right\|
={\left[\int\limits_{\sigma(A)} {\left|e^{t\lambda}\right|}^2\,d(E_A(\lambda)f,f)\right]}^{1/2} \\
&={\left[\int\limits_{\sigma(A)} e^{2t\Rep\lambda}\,d(E_A(\lambda)f,f)\right]}^{1/2}
\le {\left[e^{2s(A)t}{\|E_A(\sigma(A))f\|}^2\right]}^{1/2} \\
\intertext{\hfill since $E_A(\sigma(A))=I$;}
&={\left[e^{2s(A)t}{\|f\|}^2\right]}^{1/2}
=e^{s(A)t}\|f\| \\
\intertext{\hfill in view of $s(A)=\omega_0$;}
&=e^{\omega_0t}\|f\|.
\end{align*}

Whence, we obtain exponential estimate \eqref{grestn}, in which $\omega_0=s(A)$ (see Corollary \ref{SBeGBcor}) and $M_0:=1$ are the best stability constants.
\end{proof} 

\section{Generalized Lyapunov Stability Theorem}

From the \textit{SBeGB Condition Corollary} (Corollary \ref{SBeGBcor}) and the \textit{Exponential Estimate for Scalar Type Spectral $C_0$-Semigroups} (Proposition \ref{GRESTP}), we arrive at the following version of the \textit{Generalized Lyapunov Stability Theorem} (Theorem \ref{GLST}) for scalar type spectral $C_0$-semigroups.

\begin{thm}[GLST for Scalar Type Spectral $C_0$-Semigroups]\label{GLSTS}\ \\
A $C_0$-Semigroup $\left\{T(t) \right\}_{t\ge 0}$ (of scalar type spectral operators) on a complex Banach space generated by a scalar type spectral operator $A$ is uniformly exponentially stable iff 
\[
s(A)<0,
\]
in which case
\begin{equation*}
\|T(t)\|\le 4M_0 e^{\omega_0t},\ t\ge 0,
\end{equation*}
where
\[
\omega_0=s(A)\quad \text{and}\quad
M_0:=\sup_{\delta \in \mathscr{B}(\C)}\|E_A(\delta)\|\ge 1.
\]
\end{thm}

For the $C_0$-semigroups of normal operators on complex Hilbert spaces, in view of the \textit{Exponential Estimate for Normal $C_0$-Semigroups} (Proposition \ref{GRESTN}), the exponential estimate in the prior statement can be refined as follows.

\begin{cor}[GLST for Normal $C_0$-Semigroups]\label{GLSTN}\ \\
A $C_0$-Semigroup $\left\{T(t) \right\}_{t\ge 0}$ (of normal operators) on a complex Hilbert space generated by a normal operator $A$ is uniformly exponentially stable iff 
\[
s(A)<0,
\]
in which case
\begin{equation*}
\|T(t)\|\le e^{\omega_0t},\ t\ge 0,
\end{equation*}
where $\omega_0=s(A)$.
\end{cor}

\section{Uniform Exponential Stability}

The following statement extends the \textit{Gearhart-Pr\"uss-Greiner Characterization} (Theorem \ref{G-P-GC}) \cite{Gearhart1978,Pruss1984,Greiner1985} to 
a Banach space setting.

\begin{thm}[Characterization of Uniform Exponential Stability]\label{CUES}\ \\
For a $C_0$-semigroup $\left\{T(t) \right\}_{t\ge 0}$ on a complex Banach space $(X,\|\cdot\|)$ with generator $A$
to be uniformly exponentially stable it is necessary and, provided spectral bound equal growth bound condition \eqref{SBeGB} holds, sufficient that
\begin{equation}\label{incles}
\left\{\lambda\in \C \,\middle|\, \Rep\lambda\ge 0\right\}\subseteq \rho(A)\quad \text{and}\quad
\sup_{\Rep\lambda\ge 0}\|R(\lambda,A)\|<\infty.
\end{equation}
\end{thm} 

\begin{proof}\

\textit{Necessity}. Suppose that a $C_0$-semigroup $\left\{T(t) \right\}_{t\ge 0}$ on a complex Banach space $(X,\|\cdot\|)$ with generator $A$ is uniformly exponentially stable. Then
\begin{equation}\label{ee}
\exists\, \omega<0,\ \exists\, M=M(\omega)\ge 1:\ \|T(t)\|\le M e^{\omega t},\ t\ge 0.
\end{equation}

Therefore, 
\[
\left\{\lambda\in \C \,\middle|\, \Rep\lambda\ge 0\right\}\subseteq \left\{\lambda\in \C \,\middle|\, \Rep\lambda>\omega\right\}\subseteq \rho(A).
\]
Also, for arbitrary $\lambda \in \C$ with $\Rep\lambda>\omega$ and $f\in X$,
\[
R(\lambda,A)f=-\int_0^\infty e^{-\lambda t}T(t)f\,dt
\]
(see, e.g., \cite{Engel-Nagel2006,Hille-Phillips}) and
\begin{multline*}
\|R(\lambda,A)f\|
=\left\|-\int_0^\infty e^{-\lambda t}T(t)f\,dt\right\|
\le \int_0^\infty \left\|e^{-\lambda t}T(t)f\right\|\,dt
\\
\shoveleft{
\le \int_0^\infty e^{-\Rep\lambda t}\|T(t)\|\|f\|\,dt
}\\
\hfill \text{by \eqref{ee}};
\\
\ \ 
\le M\int_0^\infty e^{-(\Rep\lambda-\omega) t}\,dt \|f\|
=\frac{M}{\Rep\lambda-\omega}\|f\|,
\hfill
\end{multline*}
which implies that
\[
\sup_{\Rep\lambda\ge 0}\|R(\lambda,A)\|\le \sup_{\Rep\lambda\ge 0}\frac{M}{\Rep\lambda-\omega}
= -\frac{M}{\omega}<\infty,
\]
completing the proof of the \textit{necessity}.

\textit{Sufficiency}. Let us prove this part \textit{by contradiction}.

Suppose that a $C_0$-semigroup $\left\{T(t) \right\}_{t\ge 0}$ on a complex Banach space $(X,\|\cdot\|)$ with generator $A$ is subject to \textit{spectral bound equal growth bound condition} \eqref{SBeGB} and conditions \eqref{incles} but is not uniformly exponentially stable, which, by the \textit{Generalized Lyapunov Stability Theorem} (Theorem \ref{GLST}), is equivalent to the fact that
\begin{equation}\label{sbnn}
s(A):=\sup\left\{\Rep\lambda \,\middle|\,\lambda\in\sigma(A) \right\}\ge 0.
\end{equation}

Observe that this implies, in particular, that $\sigma(A)\neq \emps$ (see Preliminaries).

In view of the inclusion
\begin{equation*}
\left\{\lambda\in \C \,\middle|\, \Rep\lambda\ge 0\right\}\subseteq \rho(A),
\end{equation*}
\eqref{sbnn} implies that
\begin{equation}\label{sb0}
s(A)=0.
\end{equation}

The operator $A$ being \textit{closed}, for an arbitrary $\lambda \in \rho(A)$,
\[
\|R(\lambda,A)\|\ge \dfrac{1}{\dist(\lambda,\sigma(A))},
\]
where
\[
\dist(\lambda,\sigma(A)):=\inf_{\mu\in \sigma(A)}|\mu-\lambda|,
\]
(see, e.g., \cite{Dun-SchI,Markin2020EOT}), which, in view of \eqref{sb0}, implies that
\begin{equation*}
\sup_{\Rep\lambda\ge 0}\|R(\lambda,A)\|
\ge \sup_{\Rep\lambda\ge 0}\dfrac{1}{\dist(\lambda,\sigma(A))} 
=\dfrac{1}{\inf_{\Rep\lambda\ge 0}\dist(\lambda,\sigma(A))}=\infty,
\end{equation*}
\textit{contradicting} \eqref{incles}.

The obtained contradiction completes the proof of the \textit{sufficiency}, and hence, of the entire statement.
\end{proof} 

\begin{rems}\
\begin{itemize}
\item Thus, the \textit{necessity} of the \textit{Gearhart-Pr\"uss-Greiner characterization of the uniform exponential stability} \cite[Theorem V.$3.8$]{Engel-Nagel2006} holds in a Banach space setting.
\item As the next example shows, the requirement that the semigroup be subject to \textit{spectral bound equal growth bound condition} \eqref{SBeGB} in the \textit{sufficiency} of the prior characterization is essential and cannot be dropped.
\end{itemize}
\end{rems} 

\begin{exmp}
As discussed in \cite[Counterexample V.$1.26$]{Engel-Nagel2006}, the left translation $C_0$-semigroup
\[
[T(t)f](x):=f(x+t),\ t,x\ge 0,
\]
on the complex Banach space
\[
X:=\left\{f\in C(\R_+) \,\middle|\, \lim_{x\to \infty}f(x)=0\ \text{and}\ \int_0^\infty|f(x)|e^x\,dx<\infty \right\}
\]
($\R_+:=[0,\infty)$) with the norm
\[
X\ni f\mapsto \|f\|:=\|f\|_\infty+\|f\|_1=\sup_{x\ge 0}|f(x)|+\int_0^\infty|f(x)|e^x\,dx
\]
is generated by the differentiation operator
\[
Af:=f'
\]
with the \textit{domain}
\[
D(A):=\left\{f\in X\,\middle|\, f\in C^1(\R_+),f'\in X \right\},
\]
for which
\begin{equation*}
\sigma(A)=\left\{\lambda\in \C \,\middle|\,\Rep\lambda\le -1 \right\},
\end{equation*}
and hence $s(A)=-1$.

The semigroup is not uniformly exponentially stable since
\[
\|T(t)\|=1,\ t\ge 0,
\]
and hence, $\omega_0=0$.

Thus,
\[
s(A)\ne \omega_0.
\]

However, the resolvent
\[
R(\lambda,A)f=-\int_{0}^{\infty}e^{-\lambda t}T(t)f\,dt,\ f\in X,
\]
of the generator $A$ exists for all $\lambda\in \C$ with $\Rep\lambda>-1$ and 
\[
\sup_{\Rep\lambda\ge 0}\|R(\lambda,A)\|<\infty
\]
(see \cite[Comments V.$3.9$]{Engel-Nagel2006}).
\end{exmp}

By  the \textit{SBeGB Condition Corollary} (Corollary \ref{SBeGBcor}) and \cite[Corollary V.$2.9$]{Engel-Nagel2006} (see also \cite[Corollary V.$2.10$]{Engel-Nagel2006}), we arrive at

\begin{cor}[Characterization of Uniform Exponential Stability]\ \\
For a $C_0$-semigroup $\left\{T(t) \right\}_{t\ge 0}$ on a complex Banach space $(X,\|\cdot\|)$ with generator $A$ to be uniformly exponentially stable it is necessary and, provided the semigroup is scalar type spectral or eventually norm-continuous, in particular eventually compact, eventually differentiable, analytic, or uniformly continuous, sufficient that
\begin{equation*}
\left\{\lambda\in \C \,\middle|\, \Rep\lambda\ge 0\right\}\subseteq \rho(A)\quad \text{and}\quad
\sup_{\Rep\lambda\ge 0}\|R(\lambda,A)\|<\infty.
\end{equation*}
\end{cor} 

On the regularity of scalar type spectral $C_0$-semigroups, see \cite{arXiv:2103.05260}.

\section{Acknowledgments}

My sincere gratitude is to Dr.~Yuri Latushkin for his reference to the \textit{Gearhart-Pr\"uss-Greiner characterization} during our conversation at the Joint Mathematics Meetings 2019, thoughtful suggestions concerning the best stability constants, and insightful remarks.


\end{document}